\documentclass[sn-mathphys]{sn-jnl}% Math and Physical Sciences Reference Style
\usepackage{amssymb,amsmath,mathrsfs, amsthm}
\usepackage{hyperref}

\usepackage{cases}

\jyear{2021}%
\theoremstyle{thmstyleone}%
\numberwithin{equation}{section}
\newtheorem{theorem}{Theorem}[section]%  meant for continuous numbers
\newtheorem{question}{Question}[section]%

\newtheorem{proposition}[theorem]{Proposition}%

\theoremstyle{thmstyletwo}%
\newtheorem{remark}[theorem]{Remark}%

\theoremstyle{thmstylethree}%
\newtheorem{definition}[theorem]{Definition}%

\begin{document}

\title[On the definition and smooth examples of weakly K\"ahler--Finsler metrics]{On the definition and smooth examples of weakly K\"ahler--Finsler metrics}

%%=============================================================%%
%% Prefix	-> \pfx{Dr}
%% GivenName	-> \fnm{Joergen W.}
%% Particle	-> \spfx{van der} -> surname prefix
%% FamilyName	-> \sur{Ploeg}
%% Suffix	-> \sfx{IV}
%% NatureName	-> \tanm{Poet Laureate} -> Title after name
%% Degrees	-> \dgr{MSc, PhD}
%% \author*[1,2]{\pfx{Dr} \fnm{Joergen W.} \spfx{van der} \sur{Ploeg} \sfx{IV} \tanm{Poet Laureate}
%%                 \dgr{MSc, PhD}}\email{iauthor@gmail.com}
%%=============================================================%%
%\author[]{\fnm{Xiaoshu} \sur{Ge}}\email{gxs@amss.ac.cn}
%\author[]{\fnm{Xiaoshu} \sur{Ge}\quad email: gxs@amss.ac.cn}
\author[]{\fnm{Xiaoshu} \sur{Ge}}
\affil{\orgdiv{Institute of Mathematics, Academy of Mathematics and Systems Science}, \orgname{Chinese Academy of Sciences}, \city{Beijing}, \postcode{100190}, \country{China}\\
E-mail: gxs@amss.ac.cn}

%\author[]{\fnm{} \sur{}$^{1,}$\footnote{Corresponding author}}\email{}

%\author[1,2]{\fnm{Third} \sur{Author}}\email{iiiauthor@gmail.com}
%\equalcont{These authors contributed equally to this work.}

%\affil{\orgdiv{Institute of Mathematics, Academy of Mathematics and Systems Science}, \orgname{Chinese Academy of Sciences},  \city{Beijing }, \postcode{100190}, \country{China}}

%\affil[2]{\orgdiv{Institute of Mathematics, Academy of Mathematics and Systems Science}, \orgname{Chinese Academy of Sciences}, \orgaddress{ \city{Beijing}, \postcode{100190}, \country{China}}}

%%==================================%%
%% sample for unstructured abstract %%
%%==================================%%

\abstract{
The generalization of the K\"ahler metric in the complex Finsler setting admits four natural versions. Among these, the strongly K\"ahler--Finsler condition is known to coincide with one of the intermediate conditions. In this paper, we establish the equivalence between the weakly K\"ahler--Finsler condition and the remaining intermediate condition. Moreover, we obtain new necessary and sufficient conditions for a weakly K\"ahler--Finsler metric to be strongly K\"ahler--Finsler. Using these criteria, we construct, for the first time, smooth examples of weakly K\"ahler--Finsler metrics which are not K\"ahler--Finsler, via the Kobayashi metrics on smooth bounded strongly convex domains.
}

%%================================%%
%% Sample for structured abstract %%
%%================================%%


%%\pacs[JEL Classification]{D8, H51}


\maketitle

\section{Introduction and main results}

Complex Finsler metric is a natural generalization of Hermitian metric, which is defined as follows.
\begin{definition}[\cite{AP}]
Let $M$ be a complex manifold, a strongly pseudoconvex complex Finsler metric on $M$ is an upper semi-continuous function $F:T^{1,0}M\rightarrow [0,+\infty)$ such that

 (1) $G:=F^2$ is smooth on 	
 $\widetilde{M}:=T^{1,0}M\setminus\{\mbox{zero section}\}$;

 (2) $F(z;\lambda v)=\vert\lambda\vert F(z;v)$ for any
 $\lambda\in\mathbb{C}$ and $(z;v)\in T^{1,0}M$.
	
(3) the matrix $\left(G_{i\overline{j}}\right):=\left(\frac{\partial^2G}{\partial
v^i\partial\overline{v^j}}\right)$ is positive definite at every $(z;v)\in \widetilde{M}$.
\end{definition}

\begin{remark}
By the homogeneous of a strongly pseudoconvex complex Finsler metric $F$, the Euler homogeneous identity provides
\[G(z;v)=G_{i\bar j}(z;v)v^i\bar v^j,\quad \forall (a,v)\in \tilde M.\]
And $G$ is smooth on the entire bundle $T^{1,0}M$ if and only if $G$ is a smooth Hermitian quadratic metric (cf. \cite{AP}).
\end{remark}

In Hermitian geometry, the metrics with torsion-free Chern connection are well known as K\"ahler metrics.
For every strongly pseudoconvex complex Finsler metric $F$, there is a uniquely induced Chern--Finsler connection over the tangent bundle. According to the vanishing degree of the horizontal part
torsion, there are three known kinds of metric notions as generalizations of K\"ahler in the Finsler setting.
\begin{definition}[ \cite{AP}]
A strongly pseudoconvex complex Finsler metric $F$ is called
\begin{itemize}
\item strongly K\"ahler--Finsler if
\begin{equation}
\Gamma_{j;i}^{k}(z;v) = \Gamma_{i;j}^{k}(z;v), \quad
\forall   \, (z;v) \in \widetilde{M}, \,i,j,k = 1,\dots,n;
\label{s-Kahler}
\end{equation}
\item K\"ahler--Finsler if
\begin{equation}
(\Gamma_{j;i}^{k}(z;v) - \Gamma_{i;j}^{k}(z;v))v^j=0, \quad
\forall   \, (z;v) \in \widetilde{M}, \,i,k = 1,\dots,n;
\label{Kahler}
\end{equation}
\item weakly K\"ahler--Finsler if
\begin{equation}
G_k(z;v)(\Gamma_{j;i}^{k}(z;v)-\Gamma_{i;j}^{k}(z;v))v^j, \quad
\forall   \, (z;v) \in \widetilde{M}, \,i = 1,\dots,n,
\label{w-Kahler}
\end{equation}
where $\Gamma^k_{i;j}$ is the horizontal Chern-Finsler connection coefficient.
\end{itemize}
\end{definition}

\begin{remark}
When $F$ is a Hermitian, $\Gamma_{i;j}^{l}(z;v) = \Gamma_{i;j}^{l}(z)$ is just the coefficient of the Chern connection of $G$, and the above conditions all reduce precisely to the K\"ahler condition.
\end{remark}

Notice that these conditions are general setting in complex Finsler geometry, it is very important to understand the concise relations among the three K\"ahler type conditions.
By definition, strongly K\"ahler--Finsler implies K\"ahler--Finsler, and K\"ahler--Finsler implies weakly K\"ahler--Finsler. Surprisingly, in \cite{CS}, B. Chen and Y. Shen proved the coincidence between strongly K\"ahler--Finsler and K\"ahler--Finsler.
\begin{theorem}[\cite{CS}]\label{CHSH}
The K\"ahler--Finsler metrics are actually strongly K\"ahler.
\end{theorem}

In \cite{CGZhou}, N. Cui, J. Guo and L. Zhou constructed some weakly K\"ahler--Finsler examples with singularity, which are not K\"ahler--Finsler.
So the notions of weakly K\"ahler--Finsler and K\"ahler--Finsler should be different in general.

Notice that in \eqref{w-Kahler}, the weakly K\"ahler-Finsler condition, involve the contraction of $\Gamma_{j;i}^{k}(z;v)-\Gamma_{i;j}^{k}(z;v)$ with $G_k$ and $v^j$. The K\"ahler-Finsler condition is obtained by remove the term $G_k$, if we replace this by removing the other term, then another nature condition between strongly K\"ahler--Finsler and weakly K\"ahler--Finsler appears:
\begin{equation}
G_k(z;v)(\Gamma_{j;i}^{k}(z;v)-\Gamma_{i;j}^{k}(z;v))=0, \quad
\forall  \,(z;v) \in \widetilde{M}, \,i,j = 1,\dots,n.
\label{n-KF}
\end{equation}
This new condition is also considered at the first time, as known as the author's knowledge, when the author and Xiangyu Zhou were considering the question wether the weakly K\"ahler--Finsler Kobayashi metrics of smooth bounded strongly convex domains satisfy  the new condition \eqref{n-KF}. The case of complex dimension two was obtained in their unpublished manuscript, by observing that $G_k(\Gamma_{j;i}^{k}-\Gamma_{i;j}^{k})v^j=0$ implies
\[G_k(\Gamma_{1;2}^{k}-\Gamma_{2;1}^{k})v^1
=G_k(\Gamma_{1;2}^{k}-\Gamma_{2;1}^{k})v^2=0,\]
thus $G_k(\Gamma_{1;2}^{k}-\Gamma_{2;1}^{k})=0$ by $v\neq0$.

The author presented related results at the Young Mathematicians Workshop on Several Complex Variables, held in Incheon, Korea, in August 2026. And they proposed the following nature questions:
\begin{question}\label{que1}
(1) Does the Kobayyashi metrics of smooth bounded strongly convex domains satisfy  the new condition  \eqref{n-KF}?

(2) For complex dimension $n>2$, wether the new condition  \eqref{n-KF} coincides with the weakly K\"ahler--Finsler condition ?

(3) Wether the new condition  \eqref{n-KF} coincides with the K\"ahler--Finsler condition?
\end{question}

The following Theorem \ref{1thm} gives a positive answer to (2) of Question \ref{que1}, thus the answer to the special case, question (1), is also true.
\begin{theorem}[cf. Theorem \ref{thm}]\label{1thm}
The weakly K\"ahler--Finsler condition coincides with
\begin{equation*}
G_k(z;v)(\Gamma_{j;i}^{k}(z;v)-\Gamma_{i;j}^{k}(z;v))=0, \quad
\forall   \, (z;v) \in \widetilde{M}, \,i,j= 1,\dots,n,
\end{equation*}
which coincides with
 \begin{eqnarray*}
\delta_cG_{b}-\delta_bG_{c}=0,\quad\forall  \, (z;v) \in \widetilde{M}, \,b,c = 1,\dots,n,
 \end{eqnarray*}
\end{theorem}
 \begin{remark}
\begin{itemize}
  \item The new condition \eqref{n-KF} appears very naturally as an intermediate condition between strongly K\"ahler--Finsler and weakly K\"ahler--Finsler;
  \item known results under K\"ahler-Finsler setting may be generalized to the weakly K\"ahler-Finsler setting, and conclusions under weakly K\"ahler-Finsler setting may be stronger,  with the help of the new equivalent condition.
 \item This theorem, for complex dimension \(n>2\), was obtained using the Danus automated reasoning system.
     \end{itemize}
\end{remark}

Here we give some new sufficient and necessary conditions, for weakly K\"ahler--Finsler metrics to be K\"ahler--Finsler by using $\Gamma^i_{\bar j;k},$ the mixed (1,1)-torsion of the Chern--Finsler connection of $F$.
\begin{theorem}[cf Theorem \ref{tor}]
A weakly K\"ahler--Finsler metric is K\"ahler--Finsler if and only if one of the following holds:
 \begin{equation}
G_i(z;v)\dot\partial_j\Gamma^i_{\bar l;k}(z;v)v^k=0,
 \quad \forall\, (z;v) \in \widetilde{M}, \,j,l = 1,\dots,n;\label{4}
\end{equation}
   \begin{equation}
G_{ij}(z;v)\Gamma_{\bar l;k}^{i}(z;v)v^k=0, \quad
\forall  \, (z;v) \in \widetilde{M}, \,j,l = 1,\dots,n;\label{w+new}
\end{equation}
 \begin{equation}
G_i(z;v)\dot\partial_j\Gamma^i_{\bar l;k}(z;v)v^k
=G_{ij}(z;v)\Gamma_{\bar l;k}^{i}(z;v)v^k, \quad
\forall  \, (z;v) \in \widetilde{M}, \,j,l = 1,\dots,n.
\end{equation}
\end{theorem}
\begin{remark}
\begin{itemize}
  \item The condition \eqref{4} was due to B. Chen and Y.Shen in \cite{CS}.
  \item Condition \eqref{w+new} can be writen as $G_{ij}\dot\partial_{\bar l}\mathbb G^i=0,$ where $\mathbb G^i(z;v):=\Gamma_{;k}^{i}(z;v)v^k$ is the complex geodesic spray.
\end{itemize}
\end{remark}

Due to L. Lempert's results in \cite{Lem81}, the Kobayashi metric $F_K$ and the Carath\'eodory metric $F_C$ of every smooth bounded strongly convex domain coincide with each other. They have constant holomorphic sectional curvature $-4$.
Patrizio \cite{Pat} find that $F_K$ is a strongly
convex metric.
And by Corollary 3.4 of M. Abate and G. Patrizio \cite{AP2}, $F_K$ is a weakly K\"ahler--Finsler metric.
Recently, Gaussier and Zimmer proved that, for a bounded strongly pseudoconvex domain with $C^2$-smooth boundary, if the Kobayashi metric $F_K$ is K\"ahler, the domain must be quotient of the unit ball \cite{Gau-Zim}; similarly, X. Dong, R. Wang and B. Wong proved that, if the Carath\'eodory metric $F_C$ is K\"ahler, the domain must be the unit ball \cite{DWW}. Thus in general, the Kobayashi and Carath\'eodory metrics are not K\"ahler. X. Ge and X. Zhou obtained that the conditions of  \cite{Gau-Zim,DWW} can be relaxed to K\"ahler-Berwald \cite{GZ}.
Thus there is an important and nature question:
\begin{question}\label{Kob-ques}
 Wether the Kobayashi metric of smooth bounded strongly convex domains, which do not holomorphically equivalent to the unit ball,  are K\"ahler-Finsler?
\end{question}
\begin{remark}
 Both positive and negative examples about this question will provide important metrics: the positive case will provide non-K\"ahler--Berwald,  K\"ahler--Finsler example, and the negative case  will provide non-K\"ahler--Finsler, smooth weakly K\"ahler--Finsler example. The difficulty is that, to date,
  besides domains biholomorphic to the unit balls where the Kobayashi metrics reduce to K\"ahler, the only known strictly convex (not strongly convex) domains for which the Kobayshi metrics have been explicitly calculated is the ellipsoid
$E_m:=\{z\in \mathbb C^2:\vert z_1\vert^2+\vert z_2\vert^4<1\}$ whose Kobayashi metric is proved to be $C^1$ by B.E. Blank, D. Fan, D. Klein, S.G. Krantz, D. Ma and
M.-Y. Pang in \cite{ellipsoid}.
 \end{remark}
Using the following complete bounded strongly convex Reinhardt domains with algebraic boundary
\[\Omega:=\{z\in \mathbb C^2: \vert z_1\vert^2+\vert z_2\vert^2+t\vert z_1\vert^4<1 \},\,t>0,\]
we give a negative answer to Question \ref{Kob-ques}, and provided, for the first time, smooth strongly convex weakly K\"ahler-Finsler metrics which are not K\"ahler-Finsler.
\begin{theorem}[cf. Theorem \ref{Kob-nKF}]\label{1Kob-nKF}
The Kobayashi metric of $\Omega$ is not K\"ahler-Finsler.
\end{theorem}

\section{Complex Finsler geometry}\label{Sec-cpx-F}
	
In this section, we briefly review some basic notions of complex Finsler geometry. 	
For a strongly pseudoconvex complex Finsler manifold $(M,F)$, denoting the inverse matrix of $(G_{i\bar{j}})$ by $(G^{\bar{j}l})$.
Denoted by $\mathcal V:=ker (d\pi)\subset T^{1,0}\widetilde M $ for the vertical subbundle over $\widetilde M$, where $\pi: \widetilde M \rightarrow M$ is the canonical projection.

Since $F$ is strongly pseudoconvex, thus $\left(G_{i\bar j}\right)$ is positive definite, hence induces a Hermitian metric on $\mathcal V:$
\begin{equation}
\langle V,W\rangle:=G_{a\bar b}(z;v)V^a\bar W^b,\quad\forall V,W\in \mathcal V_{(z;v)}.\label{vert-metric}
\end{equation}
We always use the Einstein convention in this paper.
And we adopt the convention of using a subscript to indicate differentiation of $G$ and a semicolon to distinguish derivatives with respect to the base coordinates from those with respect to the fiber coordinates. For instance, $G_{k\bar{j};i} := \dfrac{\partial^3 G}{\partial z^i \partial \bar{v}^j\partial v^k}$.
And we denote
\[\partial_i:=\frac{\partial}{\partial z^i},\,
 \dot\partial_i:=\frac{\partial}{\partial v^i},
 \partial_{\bar i}:=\frac{\partial}{\partial \bar z^i},\,
 \dot\partial_{\bar i}:=\frac{\partial}{\partial \bar v^i}.\]
The Chern--Finsler connection of $F$ is defined as follows.

\begin{definition}[\cite{AP}]\label{Abate-Chern}
Let $F$ be a strongly pseudoconvex complex Finsler metric of $M$, and $\langle \cdot,\cdot\rangle$ is defined as \eqref{vert-metric}. Then Chern--Finsler connection is the unique complex vertical connection
$\nabla:\Gamma(\mathcal V)\rightarrow\Gamma(T^\ast_{\mathbb C}\widetilde M\otimes \mathcal V ) $ such that
\begin{equation}
X\langle V,W\rangle=\langle\nabla_X V,W\rangle+\langle V,\nabla_{\bar X} W\rangle \label{Chern--Finsler}
\end{equation}
for all $X\in T^{1,0}\widetilde M$ and $V,W\in \Gamma(\mathcal V).$
Furthermore, this connection is good.
\end{definition}

In local coordinates, the connection form of $\nabla$ is defined by
\[\omega^i_j=G^{\bar c i}\bar\partial G_{j\bar c}
=:\Gamma^i_{j;k}dz^k+\Gamma^i_{jk}\psi^k,\]
where
\begin{equation}
\Gamma_{j;k}^{i} = \dot\partial_j \Gamma_{;k}^{i},\,
\Gamma_{;k}^{i} =  G^{\bar{s}i} G_{\bar{s};k},\,
\Gamma_{jk}^{i} = G^{\bar{s}i} G_{j\bar{s}k}.\label{con-coe}
\end{equation}
And $\Gamma_{j;k}^{i}, \Gamma_{;k}^{i}$ and $\Gamma_{jk}^{i}$
are called the horizontal, nonlinear and vertical connextion coefficients of $\nabla$, respectively.
Denoting
\[\psi^i=dv^i+\Gamma^i_{;j}dz^j,\quad \delta_i=\partial_i-\Gamma^j_{;i}\dot\partial_j, \]
then $\{dz^i;\psi^i\}$ is the locally dual coframe of $\{\delta_i;\frac{\partial}{\partial v^j}\}.$
 And \(\{\delta_i,i=1,\cdots,n\}\) forms a local frame of the horizontal bundle $\mathcal H$, a sub-bundle of tangent bundle over $\widetilde M$:
\[T^{1,0}\widetilde M=\mathcal H\oplus\mathcal V.\]

 Since $G(z;v)$ is (1,1)-homogeneous with respect to $v$ and $\bar v$, by the Euler homogeneous identity, it is easy to check (cf. \cite{AP})
\begin{eqnarray}
G(z;v)&=&G_{a}(z;v)v^a=G_{\bar a}(z;v)\bar v^a;\\
0&=&G_{ab}(z;v)v^a=G_{ab\bar c }(z;v)v^a; \\ \Gamma^i_{j;k}(z;v)v^j&=&\Gamma^i_{;k}(z;v),\,
\Gamma^i_{jk}(z;v)v^j=0.\label{AP-homo}
\end{eqnarray}
Thus $G_iG^{\bar li}=\bar v^sG_{i\bar s}G^{\bar li}=\bar v^l.$
By definition,
$\Gamma^i_{j;k}=G^{\bar li}\delta_k G_{j\bar l}$,
thus
\[G_i\Gamma^i_{j;k}= G_iG^{\bar li}\delta_k G_{j\bar l}= \bar v^l\delta_k G_{j\bar l}.\]
Notice $\delta_k \bar v^l=0,$ we have
\begin{equation}
G_i\Gamma^i_{j;k}=\bar v^l\delta_k G_{j\bar l}=\delta_k (\bar v^lG_{j\bar l})=\delta_k G_{j}.\label{djG;k}
\end{equation}
Thus our new condition \eqref{n-KF} coincides with
 \begin{eqnarray}
 \delta_cG_b-\delta_bG_c&=&0,\quad\forall 0\neq v\in T_z^{1,0}M,
 \label{new-KF2}
 \end{eqnarray}
which also coincides with
 \begin{eqnarray}
 (\delta_cG_{b\bar a}-\delta_bG_{c\bar a})\bar v^a&=&0,\quad\forall 0\neq v\in T_z^{1,0}M.\label{new-KF3}
 \end{eqnarray}
And the original weakly K\"ahler-Finsler condition \eqref{w-Kahler} coincides with
 \begin{eqnarray}
(\delta_cG_{b}-\delta_bG_{c})v^c=0,\quad\forall 0\neq v\in T_z^{1,0}M.\label{wKF2}
 \end{eqnarray}
Notice $\Gamma^i_{j;k}=G^{\bar li}\delta_k G_{j\bar l}$, and
$G^{\bar li}$ is invertible, the K\"ahler--Finsler condition \eqref{Kahler} equivalents to
 \begin{eqnarray}
 \delta_cG_{b\bar a}-\delta_bG_{c\bar a}=0,\quad\forall 0\neq v\in T_z^{1,0}M,\label{KF2}
 \end{eqnarray}
and the strongly K\"ahler--Finsler condition \eqref{s-Kahler} equivalents to
 \begin{eqnarray}
(\delta_cG_{b\bar a}-\delta_bG_{c\bar a})v^c=0,\quad\forall 0\neq v\in T_z^{1,0}M.\label{sKF2}
 \end{eqnarray}

For the Chern-Finsler connection, its torsion has $(2,0)$ part and $(1,1)$ part as follows (cf. (2.3.19) of \cite{AP})
 \begin{eqnarray}
\theta&=&\theta^\sigma\otimes\delta_\sigma,\,
\theta^\sigma=\frac{1}{2}\left[\Gamma^\sigma_{\nu;\mu}
-\Gamma^\sigma_{\mu;\nu}\right]dz^\mu\wedge dz^\nu+\Gamma^\sigma_{\nu\gamma}\psi^\gamma\wedge dz^\nu;\nonumber\\
\tau&=&\tau^\alpha\otimes\frac{\partial}{\partial v^\alpha},\,
\tau^\alpha
=-\delta_{\bar \nu}\Gamma^\alpha_{;\mu}dz^\mu\wedge dz^{\bar\nu}
-\Gamma^\alpha_{\bar\beta;\mu}dz^\mu\wedge\overline{\psi^\beta},
\label{torsion}
 \end{eqnarray}
where $\Gamma^\alpha_{\bar\beta;\mu}
=\dot\partial_{\bar \beta}\Gamma^\alpha_{;\mu}$
is the mixed $(1,1)$ torsion.

The curvature operator of the Chern-Finsler connection is given by  (cf. (2.3.21) of \cite{AP})
\[\Omega_j^i=\bar\partial \omega_j^i
=R_{j;k\bar l}^i d z^k\wedge d\bar z^l
+R_{jk;\bar l}^i \psi^k\wedge d\bar z^l
+R_{j\bar k;l}^i d z^l\wedge \overline{ \psi^k}
+R_{jk\bar l}^i \psi^k\wedge \overline{\psi^l},\]
where the hv-curvature (the horizontal-vertical part) is given by
\begin{equation}
R_{j\bar k;l}^i
=-\dot\partial_{\bar k}
\Gamma^i_{j;l}-\Gamma^i_{js}\Gamma^s_{\bar k ;l}
=-\dot\partial_{j}
\Gamma^i_{\bar k;l}-\Gamma^i_{js}\Gamma^s_{\bar k ;l}.\label{curvature}
\end{equation}

\section{The stronger condition}\label{wKF}

The proof of Theorem \ref{thm}, the main result of this section for complex dimension \(n>2\), is obtained with the assistance of the Danus automated reasoning system.
The argument is fibrewise, and its central point is that the
fibre coordinates together with the Legendre covector components form a coframe in which the relevant vertical derivatives commute.
 Under local coordinates of $\widetilde M$, denote
\[D_j :=\dot\partial_{j}-G_{ja}G^{\bar ba} \dot\partial_{\bar b}.\]
Thus
\begin{equation}
\delta_k G_{j}
= \left(\partial_{k}
-\Gamma^a_{;k}\dot\partial_{a}\right)\dot\partial_{j}G
= \dot\partial_{j}G_{;k}-G^{\bar ba}G_{\bar b;k}G_{ja}
 = D_jG_{;k},\label{DjG;k}
\end{equation}
Denoting
\[T_{jk}:=G_i\left( \Gamma^i_{j;k}-\Gamma^i_{k;j}\right)
=\delta_kG_{j}-\delta_jG_{k},\]
thus by \eqref{DjG;k},
\begin{equation}
  T_{jk} = D_jG_{;k} -D_kG_{;j}.\label{key0}
\end{equation}
The weakly condition \eqref{w-Kahler} and the new condition  \eqref{n-KF} can be simplified as $T_{ij}v^i=0$ and
$T_{ij}=0$, respectively.

\begin{proposition}[adapted vertical exactness]\label{prop}
 Let $F$ be a strongly pseudoconvex complex Finsler metric. Then
\begin{equation}
  [D_i,D_j]=0,\, D_lT_{jk} +D_jT_{kl} +D_kT_{lj} = 0,\, \forall i,j,k=1,\cdots,n.\label{key}
\end{equation}
\end{proposition}
\begin{proof}
In order to prove $[D_i,D_j]=0,$ we just need to prove the vanishing of the pairs $(\omega^k,[D_i,D_j]),\,k=1,\cdots,2n,$ for some local co-frame $\{\omega^1,\cdots,\omega^{2n}\}$ of $T^\ast_{\mathbb C}M$.
Notice
\[D_iv^k=\delta^k_i,\quad D_iG_{k}=\dot\partial_{i}G_k-G_{il}G^{\bar sl}\dot\partial_{\bar s}G_k=G_{ki}-G_{il}G^{\bar sl}G_{k\bar s}=0,\]
we have the vanishing pairs
\[(dv^k, [D_i,D_j])=0,\quad (d_vG_k, [D_i,D_j])=0,\]
where $d_v$ is the vertical differential.
Notice \(d_vG_k=G_{k l}d v^l+G_{k\bar l}d\bar v^l,\) and
\[\left(\begin{array}{cccccc}
   1&\cdots &0 &  0&\cdots &0 \\
   \vdots&\ddots &\vdots & \vdots&\ddots &\vdots  \\
   0&\cdots & 1 & 0&\cdots &0  \\
   G_{11} &\cdots&G_{1n} & G_{1\bar 1}&\cdots&G_{1\bar n} \\
    \vdots&\ddots& \vdots & \vdots&\ddots& \vdots \\
   G_{nl} &\cdots&G_{nn}& G_{n\bar 1}&\cdots&G_{n\bar n}
  \end{array}\right)
  \left(\begin{array}{cccccc}
   dv^1 \\
   \vdots \\
   dv^n \\
   d\bar v^1\\
   \vdots\\
   d\bar v^n
  \end{array}\right)
  =
  \left(\begin{array}{cccccc}
   dv^1 \\
   \vdots \\
   dv^n \\
   G_{1l}dv^l + G_{1\bar l}d\bar v^l\\
   \vdots\\
   G_{nl}dv^l + G_{n\bar l}d\bar v^l
  \end{array}\right)
  =
  \left(\begin{array}{cccccc}
   dv^1 \\
   \vdots \\
   dv^n \\
   d_vG_1\\
   \vdots\\
   d_vG_n
  \end{array}\right),
\]
thus $\{ dv^1 ,\cdots, dv^n, d_vG_1, \cdots, d_vG_n\}$
 forms a local coframe, hence $[D_i,D_j]=0.$

At last, by $T_{jk} = D_jG_{;k} -D_kG_{;j}$ and $[D_i,D_j]=0$,
we have
\begin{eqnarray*}
&&D_lT_{jk} +D_jT_{kl} +D_kT_{lj}
\\&=& D_l(D_jG_{;k} -D_kG_{;j})
+D_j(D_kG_{;l} -D_lG_{;k})
+D_k(D_lG_{;j} -D_jG_{;l})\\
&=&[D_l,D_j]G_{;k}+[D_j,D_k]G_{;l}+[D_k,D_l]G_{;j}
\\&=&0.
\end{eqnarray*}
\end{proof}

Now we can prove the following equivalence.
\begin{theorem}\label{thm}
The weakly K\"ahler--Finsler condition coincides with
\begin{equation}
G_k(z;v)(\Gamma_{j;i}^{k}(z;v)-\Gamma_{i;j}^{k}(z;v))=0, \quad
\forall  \, (z;v) \in \widetilde{M}, \,i,j = 1,\dots,n,
\end{equation}
which coincides with
 \begin{eqnarray}
\delta_cG_{b}-\delta_bG_{c}=0,\quad\forall  \, (z;v) \in \widetilde{M}, \,b,c = 1,\dots,n.
 \end{eqnarray}
\end{theorem}
\begin{proof}
The new condition obviously implies weakly K\"ahler-Finsler condition. On the contrary, if $F$ is a weakly K\"ahler-Finsler metric, i.e., $T_{jk}v^j=0$, we need to prove $T_{ij}=0.$
Notice
\[0=D_i(T_{jk}v^j)=T_{ik}+D_iT_{jk}v^j,\]
we have \begin{equation}T_{ik}=-D_iT_{jk}v^j.\label{d1}\end{equation}
Notice $T_{ij}=-T_{ji}$, we have $T_{ij}v^j=0.$
Similarly, by $D_k(T_{ij}v^j)=0$, we have
\begin{equation}T_{ik}=-D_kT_{ij}v^j.\label{d2}\end{equation}
 By Proposition \ref{prop} we have
 \begin{equation}
 v^l(D_lT_{jk} +D_jT_{kl} +D_kT_{lj}) = 0.\label{d3}
 \end{equation}
Substituting \eqref{d1} and \eqref{d2} in  \eqref{d3} provides
 \[v^lD_lT_{jk} -T_{kj} -T_{kj} = 0,\]
 thus
 \begin{equation}
 2T_{kj}=v^lD_lT_{jk}.\label{d4}
 \end{equation}
Notice that both $T_{jk}$ and $G_{s}$ are $(0,1)$-homogeneous with respect to $v$ and $\bar v$, using the Euler homogeneous identity for $G_{s}$, we have
$v^lG_{ls}=0,$ thus
\[v^lD_l
=v^l\left(\dot\partial_{l}-G_{ls}G^{\bar rs}\dot\partial_{\bar r}\right)
=v^l\dot\partial_{l}.\]
Thus by \eqref{d4} and the Euler homogeneous identity for $T_{jk}$, we have
\[2T_{jk}=v^lD_lT_{jk}=v^l\dot\partial_{l}T_{jk}=0.\]
\end{proof}

\begin{remark}
Similarly, denoting $T_{jk\bar l}:=\delta_kG_{j\bar l}-\delta_kG_{j\bar l},$ we also have
\begin{equation}
   \delta_kG_{j\bar l}= D_jG_{\bar l;k},\, D_iT_{jk\bar l} +D_jT_{ki\bar l} +D_kT_{ij\bar l} = 0,\, \forall i,j,k,l=1,\cdots,n.\label{simlar}
\end{equation}
And the methods in the proof of Theorem \ref{thm} also applies for the equivalence between K\"ahler--Finsler and strongly  K\"ahler--Finsler conditions, which provides a new proof of Theorem \ref{CHSH}.
\end{remark}

\section{K\"ahler-Finsler versus weakly K\"ahler--Finsler }
Since we obtain a stronger condition for weakly K\"ahler--Finsler, it is nature to consider the relation between K\"ahler-Finsler and weakly K\"ahler--Finsler conditions.
In \cite{CGZhou}, N. Cui, J. Guo and L. Zhou provided some singular weakly K\"ahler--Finsler examples which are not K\"ahler--Finsler.
So these two notions should be different in general.
In \cite{CS}, B. Chen and Y. Shen give a sufficient condition for weakly K\"ahler--Finsler metrics to be K\"ahler--Finsler.
\begin{theorem}[\cite{CS}]\label{CS2}
If the weak hv-curvature of the Chern--Finsler connection of a complex Finsler metric $F$ vanishes identically,
then $F$ is K\"ahler--Finsler if and only if it is weakly K\"ahler-Finsler.
\end{theorem}
\begin{remark}\label{CS-iff}
Indeed, in their prove of Theorem \ref{CS2}, they just
used the condition $R^i_{j\bar k;l}v^l=0$, and there is a sufficient and necessary condition for a weakly K\"ahler--Finsler metric to be K\"ahler--Finsler:
\[G_i\dot\partial_{\bar l}\Gamma^i_{j;k}v^k=0,\]
which also can be written as
\begin{equation}
G_i\dot\partial_{j}\Gamma^i_{\bar l;k}v^k=0.\label{1tor}
\end{equation}
\end{remark}
Here we give another sufficient and necessary condition for weakly K\"ahler--Finsler metric to be K\"ahler--Finsler using the mixed (1,1)-torsion of the Chern--Finsler connection of $F$.
\begin{theorem}\label{tor-eq}
A weakly K\"ahler--Finsler metric is K\"ahler--Finsler if and only if
\begin{equation}
G_{kl}\dot\partial_{\bar i}\mathbb G^k=0, \quad
\forall  \, (z;v) \in \widetilde{M}, \,i,l = 1,\dots,n,\label{2tor}
\end{equation}
where $\mathbb G^k(z;v):=\Gamma_{;j}^{k}(z;v)v^j$ is the complex geodesic spray.
\end{theorem}
\begin{proof}
Let $F$ be a weakly  K\"ahler--Finsler metric, by Theorem \ref{thm} and \eqref{new-KF3}, we have
 \begin{eqnarray}
0=\delta_cG_{b}-\delta_bG_{c}
=(\delta_cG_{b\bar a}-\delta_bG_{c\bar a})\bar v^a,
\quad\forall 0\neq v\in T_z^{1,0}M.
 \end{eqnarray}
 Thus
 \begin{eqnarray*}
0=\dot\partial_{\bar k}\left[(\delta_cG_{b\bar a}-\delta_bG_{c\bar a})\bar v^a\right]
=\dot\partial_{\bar k}(\delta_cG_{b\bar a}-\delta_bG_{c\bar a})\bar v^a+\delta_cG_{b\bar k}-\delta_bG_{c\bar k},
 \end{eqnarray*}
hence
 \begin{eqnarray}
\delta_bG_{c\bar k}-\delta_cG_{b\bar k}
&=&\dot\partial_{\bar k}(\delta_cG_{b\bar a}-\delta_bG_{c\bar a})\bar v^a\nonumber\\
&=&\dot\partial_{\bar k}
\left(G_{b\bar a;c}-\Gamma_{;c}^dG_{b\bar ad}-G_{c\bar a;b}+\Gamma_{;b}^dG_{c\bar ad}\right)\bar v^a.\label{tor2}
 \end{eqnarray}
 Notice by the Euler homogeneous identity,
 \[0=G_{b\bar a\bar k;c}=G_{b\bar ad\bar k}\bar v^a
 =G_{b\bar a\bar k;c}=G_{c\bar ad\bar k}\bar v^a,
 \,G_{b\bar ad}\bar v^a=G_{bd},
 \,G_{c\bar ad}\bar v^a=G_{cd},\]
\eqref{tor2} reduces to
 \begin{eqnarray}
\delta_bG_{c\bar k}-\delta_cG_{b\bar k}
&=&-\dot\partial_{\bar k}\Gamma_{;c}^dG_{bd}
+\dot\partial_{\bar k}\Gamma_{;b}^dG_{cd}.\label{tor3}
 \end{eqnarray}
 By \eqref{tor3} and $G_{bd}v^b=0$, we have
 \begin{eqnarray}
\left(\delta_bG_{c\bar k}-\delta_cG_{b\bar k}\right)v^b
&=&v^b\dot\partial_{\bar k}\Gamma_{;b}^dG_{cd}
=v^b\Gamma_{ \bar v^k;b}^dG_{cd}.\label{tor4}
 \end{eqnarray}
 Thus $\left(\delta_bG_{c\bar k}-\delta_cG_{b\bar k}\right)v^b=0$ if and only if $v^b\Gamma_{ \bar v^k;b}^dG_{cd}=0.$
\end{proof}

For a  K\"ahler-Finsler metric, \eqref{1tor} and  \eqref{2tor} imply
\begin{equation}
\dot\partial_{j}\left(G_i\Gamma^i_{\bar l;k}\right)v^k
=G_i\dot\partial_{j}\Gamma^i_{\bar l;k}v^k
+G_{ij}\Gamma^i_{\bar l;k}v^k
=0.\label{null}
\end{equation}

On the other hand,
\begin{eqnarray*}
G_i\Gamma^i_{\bar l;k}
&=&G_i\dot\partial_{\bar l}\Gamma^i_{;k}\\
&=&G_i\dot\partial_{\bar l}\left(
G^{\bar ai}G_{\bar a;k}\right)\\
&=&G_i\left(-G^{\bar ab}G_{b\bar c\bar l}G^{\bar ci}G_{\bar a;k}
+G^{\bar ai}G_{\bar a\bar l;k}\right)\\
&=&-G^{\bar ab}G_{b\bar c\bar l}\bar v^{ c}G_{\bar a;k}
+\bar v^{a}G_{\bar a\bar l;k}\\
&=&0.
\end{eqnarray*}
Take action $\dot\partial_{j}$ on both sides provides
\begin{eqnarray}
G_i\dot\partial_{j}\Gamma^i_{\bar l;k}
+G_{ij}\Gamma^i_{\bar l;k}=0.\label{iff2}
\end{eqnarray}
Thus
\begin{eqnarray}
G_i\dot\partial_{j}\Gamma^i_{\bar l;k}v^k
=G_{ij}\Gamma^i_{\bar l;k}v^k \iff G_i\dot\partial_{j}\Gamma^i_{\bar l;k}v^k
=G_{ij}\Gamma^i_{\bar l;k}v^k=0.\label{iff3}
\end{eqnarray}
Thus by \eqref{iff2}, \eqref{iff3} and Theorem \ref{tor-eq}, we have the following result.
\begin{theorem}\label{tor}
A weakly K\"ahler--Finsler metric is K\"ahler--Finsler if and only if one of the following holds:
 \begin{equation}
G_i(z;v)\dot\partial_{j}\Gamma^i_{\bar l;k}(z;v)v^k=0,
 \quad \forall\, (z;v) \in \widetilde{M}, \,j,l = 1,\dots,n;\label{1}
\end{equation}
   \begin{equation}
G_{ij}(z;v)\Gamma_{\bar l;k}^{i}(z;v)v^k=0, \quad
\forall  \, (z;v) \in \widetilde{M}, \,j,l = 1,\dots,n;\label{2}
\end{equation}
 \begin{equation}
G_i(z;v)\dot\partial_{j}\Gamma^i_{\bar l;k}(z;v)v^k
=G_{ij}(z;v)\Gamma_{\bar l;k}^{i}(z;v)v^k, \,
\forall  \, (z;v) \in \widetilde{M}, \,j,l = 1,\dots,n.\label{3}
\end{equation}
\end{theorem}

\section{Weakly K\"ahler--Finsler Kobayashi metric}
In this section, we provide non-K\"ahler--Finsler, smooth weakly K\"ahler--Finsler metric using the Kobayashi metric of smooth bounded strongly convex domains. In the following we use notions $\dot \psi:=\frac{\partial \psi}{\partial\zeta}$ for $\zeta\in \mathbb D,$ $\lambda':=\frac{\partial \lambda}{\partial\varepsilon}$ for $\varepsilon\in \mathbb R,$ and we use lower indices to represent components
of points and vectors for calculation convenience.

Recall that:
for a smooth bounded strongly convex domain $\Omega$, and \((z,v) \in \widetilde{\Omega}\), a holomorphic map \(f_{z, v}: \mathbb D \to \Omega\) is called extremal with respect to \((z, v)\), if $f_{z,v}$ satisfies
\[f_{z,v}(0)=z,\,\dot f_{z,v}(0)=\lambda v\]
with \(\vert \lambda\vert\) being maximal (In case $\Omega$ is Reinhardt, we can assume $\lambda\in \mathbb R$).
In \cite{Lem81} (in France),  Lempert obtained the following properties (cf. also \cite{Lem82,Lem84}, and \cite{RWK13} for bounded convex domain case and without boundary regularity assumption), for any $(z,v)\in \widetilde \Omega$:

(L1) there is a unique extremal map $f_{z,v}:\mathbb D\to \Omega$ such that $ f(0)=z,\, \dot f(0)=\lambda v$.

(L2) $f_{z,v}$ is the unique geodesic complex curve through $(z,v)$, i.e., it is an isometry with respect to the Poincar\'e distance of $\mathbb D$ and the Kobayashi distance of $\Omega$, thus  $F_K(z;v)=\frac1\lambda$.

(L3) Smooth dependence: $(z,v)\mapsto f_{z,v}$ and
$(z,v)\mapsto\lambda(z,v)$ smoothly depend on $\widetilde \Omega.$ And $f_{z,v}$ extends smoothly to $\bar{\mathbb D}$, embedding $\bar{\mathbb D}$ into $\bar \Omega$, $f(\partial\mathbb D)\subset \partial\Omega$.

(L4) Dual-map criterion: a holomorphic immersion $\varphi:\bar{\mathbb D}\to\bar \Omega$, smooth up to the boundary, with $\varphi(\partial\mathbb D)\subset \partial \Omega$, is extremal if and only if there exists a real map $p>0$ defined (for almost every point) on $\partial\mathbb D$, such that
\[H_j(\zeta):=\zeta p(\zeta)\rho_{;j}(\varphi(\zeta)),\,j=1,2,\, \rho_{;j}=\partial_j\rho,\]
extends holomorphically to $\mathbb D$.

Especially, here we consider
\(\Omega:=\{z\in \mathbb C^2: \vert z_1\vert^2+\vert z_2\vert^2+t\vert z_1\vert^4<1 \},\,t>0,\)
a complete Reinhardt, strongly convex domain  with algebraic boundary. For $\Omega$ we have
\begin{equation}
H_j(\zeta):=\zeta p(\zeta)\rho_{;j}(\varphi(\zeta)),\,j=1,2,\, \rho_{;1}(z)=\bar z_1(1+2t\vert z_1\vert^2),\,
\rho_{;2}(z)=\bar z_2.\label{L4-H}
\end{equation}
Moreover, $p$ is unique up to a positive constant, and the inner product $\dot\varphi\cdot H:=\sum_j\dot\varphi_jH_j$ is a real constant on $\bar{\mathbb D}$ (it is holomorphic, and on $\partial\mathbb D$ it equals $p\zeta\dot\varphi\cdot \rho'$, where $\zeta\dot\varphi\cdot \rho'\in\mathbb R$ since $\frac{d}{d\theta}\rho(\varphi(e^{i\theta}))
=2\operatorname{Re}[i\zeta\dot\varphi\cdot \rho']=0$).
It is known that the Kobayashi metric $F_K$ of $\Omega$ is weakly K\"ahler--Finsler, and we will prove that $F_K$ is not K\"ahler-Finsler.
\begin{theorem}\label{Kob-nKF}
The Kobayashi metric of $\Omega$ is not K\"ahler-Finsler.
\end{theorem}
\begin{proof}
By Theorem \ref{tor}, we just need to
prove that $F_K$ does not satisfy the condition \eqref{2}, i.e., we just need to find some $(z,v)\in \widetilde \Omega$ and  $j,l$ such that
\[G_{ij}(z;v)\Gamma_{\bar l;k}^{i}(z;v)v^k\neq0.\]

\textbf{Step 1} Consider the expression of $g(v):=G(0;v)$.

For $v\neq0$ let $\lambda_0>0$ be a solution of
\[\lambda_0^2r+t\lambda_0^4r_1^2=1, \, r:=\| v\|^2, \, r_1:=\vert v_1\vert^2,\]
 and set $u=\lambda_0v,$
thus $\|u\|^2+t\vert u_1\vert^4=1$. The linear map
\[\varphi_{0}(\zeta):=\zeta u,\quad \zeta\in \mathbb D,\]
satisfies the stationary conditions of (L4):
it is a linear embedding;
 \[\rho(\zeta u)=\|u\|^2+t\vert u_1\vert^4-1=0,\quad \zeta\in \partial \mathbb D;\]
and take $p\vert_{\partial \mathbb D}\equiv1$,
notice $\varphi(\zeta)=\zeta u$, \eqref{L4-H} reduces to
\begin{equation}
H_j(\zeta)=\zeta\rho_{;j}(\zeta u)=\zeta\bar\zeta\,\bar h_j=\bar h_j,\, h:=(\mu u_1,u_2),\, \mu:=1+2t\vert u_1\vert^2,\label{hj}
\end{equation}
thus $H_j$ is constant and can be extended holomorphically to $\bar{\mathbb D}$.
By (L4), $\varphi_{0}$ is extremal, and $\varphi_{0}=f_{0,v}$ by the uniqueness in (L1), hence $F_K(0;v)=1/\lambda_0$ by (L2), i.e.
\begin{equation}
g(v):=G(0;v)=\lambda_0^{-2}=\frac{r+\sqrt{r^2+4tr_1^2}}2.\label{g(v)}
\end{equation}

\textbf{Step 2.} First-order variation of the extremal family.

Fix $v\neq0$, $w\in\mathbb C^2$. By (L3), $\varphi_\varepsilon:=f_{\varepsilon w,v}$ with small positive $\varepsilon$ near $0$, and $\lambda(\varepsilon):=F_K^{-1}(\varepsilon w;v)$ varies smoothly with respect to $\varepsilon$, thus by Taylor expansion with respect to $\varepsilon$,
\begin{equation}
\varphi_\varepsilon(\zeta)
=\zeta u+\varepsilon\psi(\zeta)+\tfrac{\varepsilon^2}{2}\chi(\zeta)
+O(\varepsilon^3),\quad
 \lambda(\varepsilon)=\lambda_0+\varepsilon\lambda'
 +\tfrac{\varepsilon^2}{2}\lambda''
+O(\varepsilon^3),\label{exp}
\end{equation}
with $\psi,\chi\in\mathcal O(\mathbb D)\cap C^\infty(\bar{\mathbb D})$, $\lambda',\lambda''\in\mathbb R$. The normalization in (L1) gives
\begin{equation}
u=\lambda_0v,\quad \psi(0)=w,\quad \dot\psi(0)=\lambda'\,v,\quad \chi(0)=0,\quad \dot\chi(0)=\lambda''\,v. \label{N}
\end{equation}
The coefficients are constrained by two identities in $\varepsilon$:

(B) $\rho(\varphi_\varepsilon(\zeta))\equiv0,\,\zeta\in \partial \mathbb D,$ since $\varphi_\varepsilon (\partial \mathbb D)\subset \partial\Omega$ by (L3);

(D) $H_j^\varepsilon(\zeta)=\zeta p_\varepsilon(\zeta)\rho_{;j}(\varphi_\varepsilon(\zeta))$ holomorphic in $\zeta$, with
\[p_\varepsilon(\zeta)
=1+\varepsilon q(\zeta)+O(\varepsilon^2),\quad  \zeta\in\partial\mathbb D,\]
 as real function, and $\sum_{j=1}^2\dot\varphi_{\varepsilon j}H^\varepsilon_j \equiv\mathrm{const}$ (from (L4)).

By (B), and \eqref{L4-H}-\eqref{exp},
  \begin{equation}
0=\frac{d}{d\varepsilon}\big\vert_0
[\rho(\varphi_\varepsilon(\zeta))]
=2\operatorname{Re}\bigg(\bar\zeta\sum_{j=1}^2\bar h_j\psi_j\bigg),\quad \zeta\in\partial \mathbb D, \label{B1}
\end{equation}
the infinitesimal tangency of $\psi$ to $\partial \Omega$ along $\varphi_0(\partial\mathbb D)$.
Writing
\[\psi_j=\sum_{k\ge0}\phi_{jk}\zeta^k,\,
\phi_{jk}:=\frac{1}{k!}\frac{\partial^k \psi_j}{\partial \zeta^k}(0),\,
b_k:=\bar h\cdot (\phi_{1k},\phi_{2k})\]
then
\[\bar\zeta\bar h\cdot\psi
=\bar\zeta\sum_{k\ge0}b_k\zeta^k
=b_0\bar\zeta+\sum_{k\ge1}b_k\zeta^{k-1},\quad \zeta\in\partial \mathbb D.\]
Thus \eqref{B1} implies
\begin{equation}
\operatorname{Re}b_1=0,\quad b_2=-\bar b_0,\quad b_k=0,\, k\ge3.\label{dl1}
\end{equation}
By \eqref{N} and \eqref{hj},
\begin{equation}
b_1:=\bar h\cdot\dot\psi(0)=\lambda'\bar h\cdot v
=\lambda'(\mu \bar u_1v_1+\bar u_2v_2)=\lambda'\lambda_0\big(r+2tr_1\vert u_1\vert^2\big).\label{dl2}
\end{equation}
Notice $\lambda$ is real, by \eqref{dl1} and \eqref{dl2},
\begin{equation}
\lambda'=0,\quad\text{hence}\quad \dot\psi(0)=\lambda' v=0,\, \psi=w+\frac{1}{2}\ddot\psi(0)\zeta^2+O(\vert\zeta\vert^3).\label{psi1}
\end{equation}

Take the first-order expansion of (D) with respect to $\varepsilon$:
\[H_1^\varepsilon=\bar h_1+\varepsilon\big(q\bar h_1+\zeta\kappa\bar\psi_1+2t\bar u_1^2\bar\zeta\psi_1\big)+O(\varepsilon^2),\, H_2^\varepsilon=\bar h_2+\varepsilon\big(q\bar h_2+\zeta\bar\psi_2\big)+O(\varepsilon^2),\]
where
$\kappa:=1+4t\vert u_1\vert^2$. Denote \[q(\zeta)=\sum_{m\in\mathbb Z}q_m\zeta^m,\quad \zeta\in \partial\mathbb D,\]
 then $q_{-m}=\overline{q_m}$ since $q(\zeta)$ is real.
Vanishing of the $\bar\zeta^m$-modes ($m\ge2$) gives
\begin{equation}
\bar \phi_{1,m+1}=-\bar h_1q_{-m}/\kappa,\, \bar \phi_{2,m+1}=-\bar h_2q_{-m}.\label{phijm}
\end{equation}
Since $b_{m+1}=0,$
\[q_m\left(\frac{\mu^2\vert u_1\vert^2}{\kappa}+\vert u_2\vert^2\right)=0.\]
Notice $S:=\mu^2\vert u_1\vert^2+\kappa\vert u_2\vert^2>0,$
we have
\[q_m=0,\, m\ge2.\]
 Hence $\phi_{jk}=0$ for $k\ge3$ by \eqref{phijm}; the remaining modes ($\bar\zeta^1$-modes of $H_1^\varepsilon,H_2^\varepsilon$; the $\zeta^1$-modes of \eqref{B1} and of the pairing variation; \eqref{N}) form  a linear system for
 $(c_1,c_2,q_{-1})$, where $c_j=\ddot\psi_j(0)$,
    whose determinant is a positive multiple of $S$, with unique solution
\[\psi(\zeta)=w+\frac{1}{2}\ddot\psi(0)\zeta^2,
\quad \ddot\psi_1(0)=-2\frac{u_1^2(2t\vert u_2\vert^2+\mu^2)\,\bar w_1+\mu u_1u_2\,\bar w_2}{S},\]
\[\ddot\psi_2(0)=-2\frac{ u_2\big(\mu^2u_1\,\bar w_1+\kappa u_2\,\bar w_2\big)}{S}.\]

\textbf{Step 3.} Second-order variation and expansion of $G(z;v)$.

Order 2 of (B) gives, on $\vert\zeta\vert=1$,
\begin{equation}
\operatorname{Re}\bigg(\bar\zeta\sum_{j=1}^2\bar h_j \chi_j\bigg)
=R:=-2\operatorname{Re}\big[t\bar u_1^2\bar\zeta^2\psi_1^2\big]-\kappa\vert\psi_1\vert^2
-\vert\psi_2\vert^2. \label{B2}
\end{equation}
Its zero Fourier mode, together with \eqref{N} and \eqref{dl2}, provides
\begin{eqnarray*}
r_0&:=&\lambda''\sum_{j=1}^2\bar h_j v_j\\
&=&-2\operatorname{Re}\big(t\bar u_1^2w_1\ddot\psi_1(0)\big)-\kappa\big(\vert w_1\vert^2+\frac{1}{4}\vert \ddot\psi_1(0)\vert^2\big)-\big(\vert w_2\vert^2+\frac{1}{4}\vert \ddot\psi_2(0)\vert^2\big).
\end{eqnarray*}
By \eqref{dl2}, $\lambda''$ is uniquely determined.

Notice that $\Omega$ is a Reinhardt domain, $z\mapsto -z$ is automorphism, thus $G(-z;v)=G(z;v)$, so odd-order partial derivatives of $G$ with respect to $z$ vanish at the origin. Thus by $G(\varepsilon w;v)=\lambda(\varepsilon)^{-2}$ and $\lambda'=0$,
\[G(\varepsilon w;v)=g(v)+\varepsilon^2Q(w;v)+O(\varepsilon^4),\quad Q=-\frac{\lambda''}{\lambda_0^3}.\]
Substituting $c_1,c_2$ from \textbf{Step 2} into $r_0$: $Q$ is a Hermitian form in $w$,
\[Q(w;v)=\vert w_1\vert^2A_1+\vert w_2\vert^2A_2+2\operatorname{Re}\big(w_1\bar w_2\,\bar v_1v_2\,C\big),\]
with $K:=\lambda_0^4\big(r+2tr_1\vert u_1\vert^2\big)$ and explicit coefficients (functions of $(r_1,r_2)$ only, via $\lambda_0^2=G^{-1}$, $\vert u_j\vert^2=\lambda_0^2r_j$):
\begin{equation}
A_1=\frac1K\bigg[\kappa+\frac{\kappa\vert u_1\vert^4(2t\vert u_2\vert^2+\mu^2)^2}{S^2}+\frac{\mu^4\vert u_1\vert^2\vert u_2\vert^2}{S^2}-\frac{4t\vert u_1\vert^4(2t\vert u_2\vert^2+\mu^2)}{S}\bigg],\label{eqa}
\end{equation}
\begin{equation}
A_2=\frac1K\bigg[1+\frac{\kappa \mu^2\vert u_1\vert^2\vert u_2\vert^2}{S^2}+\frac{\kappa^2\vert u_2\vert^4}{S^2}\bigg],
\end{equation}
\begin{equation}
C=\frac{\lambda_0^2\mu}{K}\bigg[\frac{\kappa\vert u_1\vert^2(2t\vert u_2\vert^2+\mu^2)+\mu\kappa\vert u_2\vert^2}{S^2}-\frac{2t\vert u_1\vert^2}{S}\bigg].\label{eqc}
\end{equation}
Since this holds for every fixed $w$ with uniform remainder, the second-order Taylor expansion of $G(\cdot\, ;v)$ at $0$ is $Q$:
\begin{equation}
G(z;v)=g(v)+\vert z_1\vert^2A_1+\vert z_2\vert^2A_2+2\operatorname{Re}\big(z_1\bar z_2\,\bar v_1v_2\,C\big)+O(\vert z\vert^3). \label{E}
\end{equation}

\textbf{Step4.}
From \eqref{E}-\eqref{eqc}, we have
\[G_{;k}(z;v)v^k=\bar z_1v_1(A_1+r_2C)+\bar z_2v_2(A_2+r_1C)+O(\vert z\vert^2).\]
Notice  $G_{ij}\Gamma^i_{\bar l;k}v^k$ only concerns with first-order derivative with respect to $z$, in the following, we restricted on the point $(z_1,0)$. Denoting
\[\alpha(r_1,r_2):=A_1(r_1,r_2)+r_2C(r_1,r_2),\]
then
\[G_{;k}((z_1,0);v)v^k=\bar z_1v_1\alpha+O(\vert z_1\vert^2).\]
Since $G^{\bar a i}(z;v)=g^{\bar a i}(v)+O(\vert z\vert)$, where $g^{\bar a i}$ is the inverse of $g_{i\bar j}=\dot\partial_{i}\dot\partial_{\bar j}g$, thus
\begin{eqnarray*}
\Gamma^i_{;k}((z_1,0);v)v^k
&=&g^{\bar m i}(v)\dot\partial_{\bar m}\big(G_{;k}((z_1,0);v)v^k\big)
+O(\vert z_1\vert^2)\\
&=&g^{\bar m i}(v)\bar z_1v_1\,\dot\partial_{\bar m}\alpha
+O(\vert z_1\vert^2).
\end{eqnarray*}
Also we have $
G_{ij}(z;v)=g_{ij}(v)+O(\vert z\vert),$ thus
\begin{equation}
G_{ij}((z_1,0);v)\dot\partial_{\bar l}\Gamma^i_{;k}((z_1,0);v)v^k
=\bar z_1v_1
g_{ij}\,\dot\partial_{\bar l}\bigg(g^{\bar m i}\dot\partial_{\bar m}\alpha\bigg)+O(\vert z_1\vert^2). \label{A}
\end{equation}
Every ingredient in \eqref{A} is an explicit elementary function of $(v,t)$.

\textbf{Step 5.} Exact evaluation in special directions.
Denote  $s:=\sqrt{1+t}>1$, and take \[v^\ast=\left(\sqrt{\frac1{1+s}},\sqrt{\frac1{1+s}}\right),\,
\frac1{1+s}=\frac1{1+\sqrt{1+t}}=\frac{\sqrt{1+t}-1}t,\]
then $r_1(v^\ast)=r_2(v^\ast)=\frac1{1+s}=\frac{\sqrt{1+t}-1}t,$
\[(r+tr_1^2)(v^\ast)
=\frac{2(1+s)+t}{(1+s)^2}
=\frac{2(1+\sqrt{1+t})+t}{(1+\sqrt{1+t})^2}=1.\]
Thus $v=u\in\partial\Omega,$ and
$\lambda_0(v)=1$ is the positive solution of $\lambda^2r+t\lambda^4r_1^2=1,$ hence $g(v)=1.$
%All the following become rational functions of $s$:
%\[\mu=1+2t\vert u_1\vert^2,\, \mu(v^\ast)
%=1+\frac{2t}{1+\sqrt{1+t}}=2\sqrt{1+t}-1=2s-1,\]
%\[\kappa=1+4t\vert u_1\vert^2,\,
%\kappa(v^\ast)=1+\frac{4t}{1+\sqrt{1+t}}=4s-3,\]
%\[S=\kappa\vert u_2\vert^2+\mu^2\vert u_1\vert^2,\,
%S(v^\ast)=(\kappa+\mu^2)r_1=\frac{2(2s^2-1)}{1+s},\]
%\[K=\lambda_0^4(r+2tr_1\vert u_1\vert^2),\,
%K(v^\ast)=r+2tr_1^2=2r_1(1+tr_1)=\frac{2s}{1+s}.\]
%\[v_1v_mg_{i1}\left(\dot\partial_{\bar 1}g^{\bar m i}\alpha_{r_m}+g^{\bar m i}v_1\alpha_{r_mr_1}\right).\]
Notice $\alpha$ depends only on $r_i$, we need to calculate
\[\left[v_1 g_{i1}\dot\partial_{\bar 1}\left(g^{\bar m i}\dot\partial_{\bar m}\alpha\right)\right]_{\vert v=v^\ast}
=\left[v_1 g_{i1}\left(\dot\partial_{\bar 1}g^{\bar m i}v_m\frac{\partial\alpha}{\partial r_m}
+g^{\bar m i}v_mv_1\frac{\partial^2\alpha}{\partial r_m\partial r_1}\right)\right]_{\vert v=v^\ast}.\]
By \eqref{g(v)},
\(g(v):=G(0;v)=\lambda_0^{-2}=\frac{r+\sqrt{r^2+4tr_1^2}}2\),
by direct calculation,
\[g_{11}
=\frac{\bar v_1^2}{2\sqrt{r^2+4tr_1^2}}\left\{
1+4t-\frac{\left[(1+4t)r_1+r_2\right]^2}{r^2+4tr_1^2}\right\}
=\frac{2tr_2^2\bar v_1^2}{(r^2+4tr_1^2)^{\frac32}},\]
\[g_{21}
=\frac{\bar v_1\bar v_2}{2\sqrt{r^2+4tr_1^2}}\left\{
1-\frac{\left[(1+4t)r_1+r_2\right]r}{r^2+4tr_1^2}\right\}
=-\frac{2tr_1r_2\bar v_1\bar v_2}{(r^2+4tr_1^2)^{\frac32}},\]
\[g_{11}(v^\ast)=-g_{12}(v^\ast)=\frac{t}{4(1+t)^{3/2}}
=\frac{s^2-1}{4s^3}.\]
similarly,
\begin{eqnarray}
(g_{i\bar j})
&=&\begin{pmatrix}\frac12+\frac{r+4tr_1}
{2(r^2+4tr_1^2)^{\frac12}}
+\frac{2tr_1r_2^2}{(r^2+4tr_1^2)^{\frac32}} & -\frac{2tr_1r_2\bar v_1v_2}{(r^2+4tr_1^2)^{\frac32}}\\
-\frac{2tr_1r_2v_1\bar v_2 }{(r^2+4tr_1^2)^{\frac32}}& \frac12+\frac{r}{2(r^2+4tr_1^2)^{\frac12}}
+\frac{2tr_1^2r_2}{(r^2+4tr_1^2)^{\frac32}}\end{pmatrix},\nonumber\\
(g_{i\bar j})(v^\ast)
&=&\frac{s+1}{4s^3}
\begin{pmatrix}
4s^3-2s^2+s-1 & -s+1\\[1mm] -s+1& (s+1)(2s-1)\end{pmatrix}.
\end{eqnarray}
Notice
\(\det\left(g_{i\bar j}(v^\ast)\right)=(g_{1\bar1}g_{2\bar2}-\vert g_{1\bar2}\vert^2)(v^\ast)=\frac{(s+1)^2(2s^2-1)}{4s^3},\)
thus
\[\left(g^{\bar m i}(v^\ast)\right)
=\frac1{(s+1)(2s^2-1)}
\begin{pmatrix}
(s+1)(2s-1)&s-1\\s-1&4s^3-2s^2+s-1
\end{pmatrix}.
\]

\[\begin{aligned}
g_{1\bar1\bar1}&=\frac{2tr_2^2v_1}{(r^2+4tr_1^2)^{\frac52}}
\left(2(r^2+4tr_1^2)-3r_1(r+4tr_1)\right),\\
g_{1\bar2\bar1}&=\frac{2tr_1r_2v_2}{(r^2+4tr_1^2)^{\frac52}}
\left(3r_1(r+4tr_1)-2(r^2+4tr_1^2)\right),\\
g_{2\bar1\bar1}&=\frac{2tr_2v_1^2\bar v_2}{(r^2+4tr_1^2)^{\frac52}}
\left(3r_1(r+4tr_1)-(r^2+4tr_1^2)\right),\\
g_{2\bar2\bar1}&=\frac{2tr_1r_2v_1}{(r^2+4tr_1^2)^{\frac52}}
\left(r^2+4tr_1^2-3r_1(r+4tr_1)\right),
\end{aligned}\]
and
\[(g_{i\bar m\bar1})(v^\ast)
=\frac{(s-1)(s+1)^{3/2}}{8s^5}
\begin{pmatrix}3-2s^2 & 2s^2-3 \\ 4s^2-3 & 3-4s^2\end{pmatrix}.\]
In the expression of $\alpha=A_1+r_2C$ with \eqref{eqa} and \eqref{eqc}, recall
\[\mu=1+2t\vert u_1\vert^2, \,
\kappa=1+4t\vert u_1\vert^2,\,
S=\kappa\vert u_2\vert^2+\mu^2\vert u_1\vert^2,\,
K=\lambda_0^4(r+2tr_1\vert u_1\vert^2),\]
 we should use $\vert u_i\vert^2=\lambda_0^2r_i=\frac{r_i}{g}
=\frac{2r_i}{r+\sqrt{r^2+4tr_1^2}}.$
By direct calculation,
\begin{eqnarray*}
\frac{\partial A_1}{\partial r_1}(v^\ast)
&=&\frac{(s+1)^2
(64s^7-64s^6-32s^5+44s^4-8s^3-3s^2+6s-3)}{8s^3D^2},\\
\frac{\partial A_1}{\partial r_2}(v^\ast)
&=&\frac{(s+1)^2(16s^6-16s^5-8s^4+16s^3-3s^2-6s+3)}{8s^3(2s^2-1)^2},\\
\frac{\partial^2A_1}{\partial r_1^2}(v^\ast)
&=&-\frac{(s-1)^2(s+1)^3(48s^6-84s^4+50s^2-9)}{16s^5(2s^2-1)^3},\\
\frac{\partial^2A_1}{\partial r_2\partial r_1}(v^\ast)
&=&\frac{(s-1)^2(s+1)^3(48s^6-84s^4+50s^2-9)}{16s^5(2s^2-1)^3};\\
\frac{\partial C}{\partial r_1}(v^\ast)
&=&\frac{(s-1)^2(s+1)^3(2s-1)(2s+1)}{8s^3(2s^2-1)^2},\\
\frac{\partial C}{\partial r_2}(v^\ast)
&=&-\frac{(s-1)^2(s+1)^3(2s-1)(2s+1)}{8s^3(2s^2-1)^2},\\
\frac{\partial^2C}{\partial r_1^2}(v^\ast)
&=&-\frac{(s-1)^2(s+1)^4(4s^2-3)(12s^4-6s^2+1)}{16s^5(2s^2-1)^3},\\
\frac{\partial^2C}{\partial r_2\partial r_1}(v^\ast)
&=&\frac{(s-1)^2(s+1)^4(32s^6-48s^4+20s^2-3)}{16s^5(2s^2-1)^3}.
\end{eqnarray*}
Substituting them into the formulas for $\alpha$,
\begin{eqnarray*}
\frac{\partial\alpha}{\partial r_1}(v^\ast)
&=&\frac{(s+1)^2(4s^3-4s^2+2s-1)}{2s^3},\,
\frac{\partial\alpha}{\partial r_2}(v^\ast)
=\frac{(s+1)^2(2s^2-2s+1)}{2s^3},\\
\frac{\partial^2\alpha}{\partial r_1^2}(v^\ast)&=&-\frac{3(s-1)^2(s+1)^3}{4s^5},\,
\frac{\partial^2\alpha}{\partial r_1r_2}(v^\ast)
=\frac{3(s-1)^2(s+1)^3}{4s^5}.
\end{eqnarray*}
At last, substitute the above quantities into \eqref{A}, and notice $s>1$, yields
\begin{eqnarray*}
&&\left[G_{i1}((z_1,0);v)\dot\partial_{\bar 1}\Gamma^i_{;k}((z_1,0);v)v^{k}\right]_{\vert v=v^\ast}\\
&=&-\frac{(s-1)^2(s+1)^{\frac32}(4s^2-3)}{4s^3(2s^2-1)^2}\bar z_1
+O(\vert z_1\vert^2).
\end{eqnarray*}
Since $\frac{(s-1)^2(s+1)^{\frac32}(4s^2-3)}{4s^3(2s^2-1)^2}>0$ is independent of $z_1$, thus
\[\left[G_{i1}((z_1,0);v)\dot\partial_{\bar 1}\Gamma^i_{;k}((z_1,0);v)v^{k}\right]_{\vert v=v^\ast}\neq0\]
 for sufficiently small $\vert z_1\vert>0$.
\end{proof}
\begin{remark}
This is the first smooth weakly K\"ahler-Finsler example which is not K\"ahler-Finsler. And it also implies that even when a weakly K\"ahler-Finsler has constant holomorphic sectioal curvature, it may still fail to be K\"ahler-Finsler. By the same methods, one can check that the Kobayashi metrics of the strongly convex domains $\left\{z\in\mathbb C^2:\|z\|^2+t\left(\vert z_1\vert^4+\vert z_2\vert^4\right)\right\},t>0,$ are not K\"ahler--Finsler.
\end{remark}
By Theorem \ref{Kob-nKF}, Question \ref{Kob-ques} can be reduced to the following problem.
\begin{question}\label{KF-nKB}
Is there any smooth bounded strongly convex domain which does not holomorphically equivalent to the unit ball, such that the Kobayashi metric is K\"ahler--Finsler?
\end{question}
\begin{remark}
The positive answer to Question \ref{KF-nKB} will provide the first  non-K\"ahler--Berwald, K\"ahler--Finsler example.
\end{remark}

\vskip0.4cm
\noindent
\textbf{Acknowledgement.}\ {\small The author would like to thank his Postdoctoral Supervisor, Professor Xiangyu Zhou, for his encouragement of this research.

\vskip0.4cm
\noindent
\textbf{AI Use Disclosure.}\ {\small
The research direction and initial mathematical framework were developed by the author. Theorem \ref{1thm} for complex dimension $n>2,$ was obtained using the Danus System, an automated mathematical reasoning system built on Rethlas, with base model gpt-5.6-sol.
The examples and Theorem \ref{tor} used in the proof of Theorem \ref{1Kob-nKF} are due to the author, and the idea of employing Lempert's results in this proof was also proposed by the author.  The verification of these examples, which do not satisfy the conditions in Theorem \ref{tor}, was performed with assistance from the Kimi K3 large language model.

The author would also like to thank the members of the Rethlas and Danus team (namely Leheng Chen, Guoxiong Gao, Jiedong Jiang, Haocheng Ju, Jihao Liu, Shurui Liu, Zeming Sun, Yuefeng Wang, Bin Wu, Liang Xiao, and Bin Dong) for their contributions to the development of these systems.

The author is fully responsible for the correctness of the paper.

%\bibliography{sn-bibliography}% common bib file
%% if required, the content of .bbl file can be included here once bbl is generated
%%\input sn-article.bbl
%\input{New-WKF101.bbl}
%% Default %%

%

\end{document}